\documentclass{amsart}
\usepackage{blindtext}
\usepackage{amsmath, empheq, amssymb, stackrel}
\usepackage{amsthm}
\usepackage{amssymb}
\usepackage{enumitem}
\usepackage{amsfonts}
\usepackage{microtype}
\usepackage{mathrsfs}
\usepackage[utf8]{inputenc}
\usepackage[a4paper, margin=1.5in]{geometry}
\newtheorem{theorem}{Theorem}[section]
\newtheorem{proposition}[theorem]{Proposition}
\newtheorem{corollary}[theorem]{Corollary}
\newtheorem{lemma}[theorem]{Lemma}

\newtheorem{example}[theorem]{Example}
\newtheorem{remark}[theorem]{Remark}

\newcommand{\uphi}{\underline{\phi}}
\newcommand{\upsi}{\underline{\psi}}
\newcommand{\uf}{\underline{f}}
\newcommand{\uG}{\underline{G}}
\newcommand{\myarrows}[2]{\stackrel[#2]{#1}{\overrightarrow{\underleftarrow{\hspace{1cm}}}}}
\numberwithin{equation}{section}

\DeclareMathOperator{\Index}{index}
\DeclareMathOperator{\Image}{Im}


\begin{document}
\title[On the Index of Harmonic Maps into Complex Projective Spaces]{On the Index of Harmonic Maps from Surfaces to Complex Projective Spaces}
\author{Joe Oliver}
\address{School of mathematics \\ University of Leeds \\ Leeds LS2 9JT\\ UK}
\email{mmjlo@leeds.ac.uk}
\thanks{The author is grateful for the financial support of the EPSRC doctoral training grant}

\begin{abstract}
We estimate the dimensions of the spaces of holomorphic sections of certain line bundles to give improved lower bounds on the index of complex isotropic harmonic maps to complex projective space from the sphere and torus, and in some cases from higher genus surfaces. 
\end{abstract}

\maketitle

\section{Introduction}
\emph{Harmonic maps} are smooth maps between Riemannian manifolds which are critical points of the Dirichlet energy functional (see, for example, \cite{EellsLem, urakawa}). The \emph{index} of a harmonic map is a measure of its stability and is generally difficult to calculate. Examples of stable harmonic maps are constant mappings between Riemannian manifolds and holomorphic maps between K\"{a}hler manifolds, which both have $\Index$ $0$ \cite{urakawa}. Any harmonic map from a compact Riemann surface to complex projective space which is neither holomorphic or antiholomorphic is unstable \cite{BurnsBart, Burns et. al.}. Therefore, harmonic maps given by the \emph{Gauss transform} of a full holomorphic map from a Riemann surface of genus $g$ to complex projective space that are neither holomorphic or antiholomorphic are unstable. These maps are of particular interest as they form a large class of harmonic maps called \emph{complex isotropic}, or equivalently of \emph{finite uniton number}; this class includes all harmonic maps from the $2$-sphere (see \cite{EellsWood}). 

We give new bounds on the index of harmonic maps from the $2$-sphere to complex projective space and complex isotropic harmonic maps from the torus to complex projective space (Theorem \ref{general-index-prop}), which improve those in \cite{EellsWood}. We also give new bounds on the index of harmonic maps from higher genus surfaces to complex projective space (Theorem \ref{general-index-prop}), which improve those in \cite{EellsWood} in some cases. This is achieved by recalling that holomorphic vector fields along a harmonic map $\phi$ give smooth variations that contribute to the index of $\phi$ \cite[p.\ 258]{EellsWood}; we improve the known estimates by calculating an estimate for the dimension of the space of holomorphic vector fields along a harmonic map $\phi$ by decomposing the tangent bundle using the harmonic sequence of $\phi$. 

In \cite{EellsWood} J. Eells and J.C. Wood classified all complex isotropic harmonic maps from a Riemann surface $M$ to $\mathbb{CP}^n$. Later F.E. Burstall and J.C. Wood gave an interpretation of this in \cite{BurWood} by considering maps from $M$ to a Grassmannian as subbundles of the trivial bundle $M \times \mathbb{C}^{n+1}$, and developing a technique of analysing harmonic maps from a Riemann surface into a complex Grassmannian using ``diagrams". In \cite{BurWood}, the harmonic maps of \cite{EellsWood} are constructed by a repeated use of a Gauss transform and we will use the interpretation in \cite{BurWood} to calculate bounds on the index of harmonic maps constructed in this way.

The author thanks the referee for very useful comments on this paper. The author also thanks John C. Wood for his helpful guidance and support during the preparation of this article.

\section{Preliminaries}
We recall the construction of complex isotropic harmonic maps given in \cite{BurWood, EellsWood}; for additional reading related to these constructions see \cite{Bolton, Craw, LemWood2, LemWood1, Wolfson2} and for a moving frames approach see \cite{ChernWolfson, Wolfson1}.

\subsection{Subbundles of $M \times \mathbb{C}^{n+1}$} 
Let $M$ be a compact Riemann surface. Let us identify $\mathbb{CP}^n$ with the set of (complex) lines (i.e. one-dimensional complex subspaces in $\mathbb{C}^{n+1}$) in the usual way, so that each point $V \in \mathbb{CP}^n$ is identified with a line in $\mathbb{C}^{n+1}$. The tautological bundle $T$ over $\mathbb{CP}^n$ is the subbundle of the trivial bundle $\mathbb{CP}^n \times \mathbb{C}^{n+1}$ whose fibre at $V \in \mathbb{CP}^n$ is the line $V$ in $\mathbb{C}^{n+1}.$
By decomposing the complexified tangent bundle $T^{\mathbb{C}}\mathbb{CP}^n$ using the complex structure in the usual way we have
$$T^{\mathbb{C}}\mathbb{CP}^n=T^{(1,0)}\mathbb{CP}^n \oplus T^{(0,1)}\mathbb{CP}^n.$$
There is a well-known connection-preserving isomorphism $h:T^{(1,0)}\mathbb{CP}^n \rightarrow L(T,T^{\bot})$ given by
\begin{equation}\label{iso}
h(Z)\sigma=\pi_{T^{\bot}}Z(\sigma)
\end{equation}
where $\sigma$ is a local section of  $T$, $Z \in T^{(1,0)}\mathbb{CP}^n$, $\pi_{T^{\bot}}$ denotes the orthogonal projection onto $T^{\bot}$ and $Z(\cdot)$ denotes differentiation with respect to $Z$  \cite{BurWood, Bolton, EellsWood}.

Consider a smooth map $\phi:M \rightarrow \mathbb{CP}^n$. We may decompose the $\mathbb{C}$-linear extension of its differential $d\phi$ into components
$$\partial\phi :T^{(1,0)}M\rightarrow T^{(1,0)}\mathbb{CP}^n, \qquad \overline{\partial}\phi :T^{(0,1)}M\rightarrow T^{(1,0)}\mathbb{CP}^n.$$

To each map $\phi:M\rightarrow \mathbb{CP}^n$, we may associate the pullback of the tautological bundle $\uphi:=\phi^{-1}T$; this is the complex line subbundle of the trivial bundle $M \times \mathbb{C}^{n+1}$ over $M$ whose fibre at $z$ is the line $\phi(z)$. The orthogonal projection $\pi_{\phi}$ onto $\uphi$ applied to the standard derivation on $M \times \mathbb{C}^{n+1}$ over $M$ induces a connection $^\phi\nabla$ on $\uphi$; on a (local complex) chart $(U,z)$ of $M$ this is given by
\begin{equation}\label{connection-subbundles-def}
^\phi\nabla_{\partial \slash \partial z}\upsilon=\pi_{\phi}\frac{\partial}{\partial z}\upsilon, \qquad ^\phi\nabla_{\partial \slash \partial \overline{z}}\upsilon=\pi_{\phi}\frac{\partial}{\partial \overline{z}}\upsilon,
\end{equation}
for $\upsilon \in \Gamma(\phi^{-1}T)$.

Given mutually orthogonal subbundles $\uphi$ and $\upsi$, as in \cite{BurWood} we define the bundle maps $A'_{\phi,\psi}:\uphi\rightarrow \upsi$ and $A''_{\phi,\psi}:\uphi \rightarrow \upsi$ by
$$A'_{\phi,\psi}(\upsilon)=\pi_{\psi}\frac{\partial}{\partial z}\upsilon \quad \text{and} \quad A''_{\phi,\psi}(v)=\pi_{\psi}\frac{\partial}{\partial \overline{z}}\upsilon,$$
where $\pi_{\psi}$ is the orthogonal projection onto $\upsi$ (some authors \cite{BurSal} interchange $\phi$ with $\psi$ in this notation).
These two maps are adjoint up to sign \cite{BurSal}, more concretely, with $\langle \ ,\ \rangle_{\phi}$ the Hermitian metric on $\uphi$ induced from the flat metric on the trivial bundle $\mathbb{CP}^n \times \mathbb{C}^{n+1}$ then
\begin{equation*}\label{adjoint}
-\langle A'_{\phi,\psi}\upsilon,w\rangle_{\phi} = \langle \upsilon, A''_{\psi,\phi}w\rangle_{\phi} \quad \text{for} \quad \upsilon \in \uphi, \ \ w \in \upsi.
\end{equation*}
A very useful special case of the above is the following, we set
\begin{equation*}\label{specialcaseA}
A'_{\phi}=A'_{\phi,\phi^{\bot}}:\uphi \rightarrow \uphi^{\bot} \quad \text{and} \quad A''_{\phi}=A''_{\phi,\phi^{\bot}}: \uphi \rightarrow \uphi^{\bot}.
\end{equation*}
Then, using the pullback of (\ref{iso}), we have the following isomorphism of bundles over $M$:
\begin{equation}\label{pullback-iso-equation}
\phi^{-1}T^{(1,0)}\mathbb{CP}^n \cong L(\uphi,\uphi^{\bot}).
\end{equation}
Again, this isomorphism is connection-preserving and can be used to identify the bundle maps $A'_{\phi}$ and $A''_{\phi}$ with $\partial\phi({\partial \slash \partial z})$ and $\overline{\partial}\phi({\partial \slash \partial \overline{z}})$ respectively \cite{BurWood}.
We give all bundles their Koszul--Malgrange structure \cite{KosMel}, i.e. that with $\bar{\partial}$-operator given by the $(0,1)$-part of the connection on the respective bundle. It follows that (\ref{pullback-iso-equation}) is an isomorphism of \emph{holomorphic} bundles.
Then we have the following for a smooth map $\phi: M \rightarrow \mathbb{CP}^n$.
\begin{lemma}\label{burlem1}\cite{BurWood}\leavevmode 
\begin{enumerate}[label={\upshape(\roman*)}]
\item The map $\phi$ is holomorphic (respectively antiholomorphic) if and only if $A_{\phi}''=0$ (respectively $A_{\phi}'=0$).
\item The map $\phi$ is harmonic if and only if $A_{\phi}':\uphi \rightarrow \uphi^{\bot}$ is holomorphic, i.e.,
$$A_{\phi}' \circ {^\phi\nabla_{\partial \slash \partial \overline{z}}}= {^{\phi^\bot}\nabla_{\partial \slash \partial \overline{z}}} \circ A_{\phi}',$$
or equivalently, $A_{\phi}'':\uphi \rightarrow \uphi^{\bot}$ is antiholomorphic.
\end{enumerate}
\end{lemma}
Let $\phi:M \rightarrow \mathbb{CP}^n$ be a non-antiholomorphic harmonic map. After a process of filling out the zeros of $A_{\phi}'$ detailed in \cite[p.\ 266]{BurWood}, according to Lemma \ref{burlem1} the image of $A_{\phi}'$ is a holomorphic subbundle of $\uphi^{\bot}$ which we denote by $\underline{\Image}{A_{\phi}'}$. We call this holomorphic subbundle the \emph{$\partial'$-Gauss bundle} and denote it by $\uG'(\phi)$; we have $\uG'(\phi):=G'(\phi)^{-1}T$ for some smooth map $G'(\phi):M \rightarrow \mathbb{CP}^n$.
Similarly let $\phi:M \rightarrow \mathbb{CP}^n$ be a nonholomorphic harmonic map, then the image of $A_{\phi}''$ is an antiholomorphic subbundle of $\uphi^{\bot}$ denoted $\uG''(\phi)$ and called the \emph{$\partial''$-Gauss bundle}. As before $\uG''(\phi):=G''(\phi)^{-1}T$ for some smooth map $G''(\phi):M \rightarrow \mathbb{CP}^n$.
\begin{lemma}\cite[Proposition 2.3 and Remark]{BurWood}\label{burlem2}
Let $\phi:M \rightarrow \mathbb{CP}^n$ be a harmonic map. If $\phi$ is not antiholomorphic then $G'(\phi)$ is harmonic and $G''(G'(\phi))=\phi$. If $\phi$ is not holomorphic then $G''(\phi)$ is harmonic and $G'(G''(\phi))=\phi$.
\end{lemma}
We shall now construct a harmonic sequence using the above. A map $f:M \rightarrow \mathbb{CP}^n$ is said to be \emph{full} if its image does not lie in a proper projective subspace of $\mathbb{CP}^n.$ Let $f_0:M \rightarrow \mathbb{CP}^n$ be a full holomorphic map, then as above $\uG'(f_0):=\underline{\Image}{A_{f_0}'} \subset \uf_0^{\bot}$ (most authors omit the underlining in the bundle $\uG'(\cdot)$; we add it for clarity) where $G'(f_0):M \rightarrow \mathbb{CP}^n$ is a harmonic map by Lemma \ref{burlem2}. Applying the procedure again to $G'(f_0):M \rightarrow \mathbb{CP}^n$ we have $\uG'(G'(f_0)):=\underline{\Image}{A_{G'(f_0)}'} \subset \uG'(f_0)^{\bot}$ where $G'(G'(f_0)):M \rightarrow \mathbb{CP}^n$ is again a harmonic map by Lemma \ref{burlem2}. 

For ease of notation write, $f_j:=G'(G'(\dots G'(G'(f_0)) \dots))$ where $G'$ is applied $j$ times to $f_0$, and $\uf_j=\uG'(G'(\dots G'(G'(f_0)) \dots))$, so $f_j:M \rightarrow \mathbb{CP}^n$ is a harmonic map and $\uf_j:=f_j^{-1}T$ its associated subbundle given by the pullback of the tautological bundle. Note that $\uf_j:=\underline{\Image}{A_{f_{j-1}}'} \subset \uf_{j-1}^{\bot}$.

\begin{remark}
Similarly, given a full antiholomorphic map $g_0$, replacing $A'$ and $G'$ with $A''$ and $G''$ respectively, we obtain $g_k:=G''(G''(\dots G''(G''(g_0)) \dots))$ where $G''$ is applied $k$ times to $g_0$.
\end{remark}

It was shown in \cite{BurWood} and through a different interpretation in \cite{EellsWood} that the $n$th iteration of the procedure above gives $\uf_n$ where $f_n:M\rightarrow \mathbb{CP}^n$ is a full antiholomorphic map. Using Lemma \ref{burlem1} we see that $A''_{f_0}=A'_{f_n}=0$ since $f_0$ and $f_n$ are holomorphic and antiholomorphic respectively, therefore $\uG''(f_0)=\uG'(f_n)=0$ and so do not define maps into $\mathbb{CP}^n$. Therefore we have the following sequence of associated subbundles of harmonic maps and bundle maps between them, called the \emph{harmonic sequence} \cite{Bolton,Wolfson2}: 
\begin{equation}\label{Adiagram}
\uf_0 \myarrows{A'_{f_0}}{A''_{f_1}} \uf_1 
\myarrows{A'_{f_1}}{A''_{f_2}} 
\dots
\myarrows{A'_{f_{\rho-2}}}{A''_{f_{\rho-1}}} \uf_{\rho-1}
\myarrows{A'_{f_{\rho-1}}}{A''_{f_{\rho}}} \uf_{\rho} 
\myarrows{A'_{f_{\rho}}}{A''_{f_{\rho+1}}} \uf_{\rho+1}
\myarrows{A'_{f_{\rho+1}}}{A''_{f_{\rho+2}}}
\dots
\myarrows{A'_{f_{n-1}}}{A''_{f_{n}}} \uf_{n},
\end{equation}
where $f_0 : M \rightarrow \mathbb{CP}^n$ is a full holomorphic map with associated bundle $\uf_0 :={f_0}^{-1}T$, $f_i : M \rightarrow \mathbb{CP}^n$ is a full harmonic map with associated bundle $\uf_i :={f_i}^{-1}T$ for each $i \in \{1,2,\dots,n-1 \}$, and $f_n : M \rightarrow \mathbb{CP}^n$ is a full antiholomorphic map with associated bundle $\uf_n :={f_n}^{-1}T$. 

Let $(U,z)$ be a chart of $M$ and let $z_0 \in U$ be a zero of $A_{f_{\rho-1}}'$ where $\rho \in \{1, \dots ,n \}$, then we can write
\begin{equation*}\label{Afzero}
A_{f_{\rho-1}}'(z)=(z-z_0)^k\lambda(z)
\end{equation*}
where $\lambda$ is a smooth section of $L(\uf_{\rho-1},\uf_{\rho-1}^{\bot})$, non-zero at $z_0$ and $k \in \mathbb{N}$. Then we say that $f_0$ is \emph{$\rho$th(-order) ramified} at the point $z_0$ with \emph{$\rho$th ramification index} $k$. We call the sum of all ramification indices of the points of $\rho$th ramification the \emph{$\rho$th total ramification index} and denote it $r_{\rho-1}$.
\begin{remark}
All harmonic maps $S^2 \rightarrow \mathbb{CP}^n$ are given as above; for higher genera the construction gives all harmonic maps which are (complex) isotropic \cite{EellsWood}, in the sense that all their Gauss bundles are mutually orthogonal, cf.\ \cite[\S 3]{BurWood}. This is equivalent to the finiteness of the uniton number, see, for example, \cite[\S 4.3]{Aleman}. The terms \emph{of infinite isotropy order}, \emph{strongly isotropic} and \emph{pseudoholomorphic} \cite{BoltonWoodward} are also used. Holomorphic and antiholomorphic maps are complex isotropic, so that the subbundles in (\ref{Adiagram}) are mutually orthogonal.
\end{remark}

\section{Estimates of the Index}
In \cite{EellsWood} an estimate was given for the index of nonholomorphic harmonic maps $\phi:M_g \rightarrow \mathbb{CP}^n$ where $M_g$ is a closed Riemann surface of genus $g$. Throughout this paper we call a map that is holomorphic or antiholomorphic \emph{$\pm$-holomorphic}. 
\begin{proposition}\cite{EellsWood}\label{estimate}
Let $\phi:M_g \rightarrow \mathbb{CP}^n$ be a non-$\pm$-holomorphic harmonic map. Then
$$\Index(\phi)\geq \deg(\phi)(n+1)+n(1-g).$$
\end{proposition}
Here $\deg(\phi)$ denotes the degree of $\phi$ on second cohomology; this is equal to minus the first Chern class of $\uphi$, see \cite[Lemma 5.1]{BurWood}.
In this section we shall give improvements to this estimate for genus 0, for complex isotropic harmonic maps for genus 1 and in some cases for complex isotropic harmonic maps for higher genus.
\subsection{The Space of Holomorphic Sections and Holomorphic Differentials}
The estimate in Proposition \ref{estimate} was constructed by noting that a holomorphic vector field along $\phi$ gives a smooth variation of $\phi$ that contributes to the index of $\phi$.
\begin{lemma}\cite[p.\ 258]{EellsWood}\label{index>dim}
Let $\phi:M_g \rightarrow \mathbb{CP}^n$ be a non-$\pm$-holomorphic harmonic map then
$$\Index(\phi) \geq \dim H^0(M_g,\phi^{-1}T^{(1,0)}\mathbb{CP}^n)$$
where $H^0(M_g,\phi^{-1}T^{(1,0)}\mathbb{CP}^n)$ is the space of holomorphic sections of\/ $\phi^{-1}T^{(1,0)}\mathbb{CP}^n$ over $M_g$.
\end{lemma}
\begin{theorem}[Riemann--Roch \cite{Gunning}]\label{RiemannRoch}
Let $W \rightarrow M_g$ be a holomorphic vector bundle of rank $n$ over Riemann surface $M_g$ of genus $g$ then
 $$\dim_{\mathbb{C}}H^0(M_g,W)-\dim_{\mathbb{C}}H^1(M_g,W)=c_1(\wedge^n W)+n(1-g)$$
 where $c_1$ is the first Chern class (evaluated on the canonical generator of $H_2(M_g,\mathbb{Z})$) and $H^k(M_g,W)$ for $k=0,1$ are the \v{C}ech cohomology groups of the holomorphic bundle $W$.
\end{theorem}
Note that $H^0(M_g,W)$ is the space of holomorphic sections of $W$ over $M_g$ and so $H^1(M_g,W)$ is (by Serre duality) the space of holomorphic $(1,0)$-forms of $M_g$ with values in the dual, $W^*$, of $W$ \cite[Theorem 9]{Gunning}.
\begin{corollary}\label{RiemannRoch-cor}
Let $W \rightarrow M_g$ be a complex vector bundle over a Riemann surface $M_g$ which can be given more than one distinct holomorphic structure (i.e. the structure of a holomorphic vector bundle which gives the complex structure), then $\dim H^0(M_g,W)-\dim H^1(M_g,W)$ is independent of the choice of holomorphic structure.
\end{corollary}
\begin{proof}
By Theorem \ref{RiemannRoch} we have $\dim H^0(M_g,W)-\dim H^1(M_g,W)=c_1(\wedge^n W)+n(1-g)$, where the right-hand side is only dependent on the complex structure.
\end{proof}
Let $\phi:M_g \rightarrow \mathbb{CP}^n$ be a non-$\pm$-holomorphic harmonic map. Then using Riemann--Roch for the holomorphic vector bundle $\phi^{-1}T^{(1,0)}\mathbb{CP}^n \rightarrow M_g$ of rank $n$ we get
$$\dim_{\mathbb{C}}H^0(M_g,\phi^{-1}T^{(1,0)}\mathbb{CP}^n)-\dim_{\mathbb{C}}H^1(M_g,\phi^{-1}T^{(1,0)}\mathbb{CP}^n)=\deg(\phi)(n+1)+n(1-g)$$
and Proposition \ref{estimate} follows directly by disregarding the non-negative number \\ $\dim_{\mathbb{C}}H^1(M_g,\phi^{-1}T^{(1,0)}\mathbb{CP}^n)$. We improve the estimate in Proposition \ref{estimate} by looking at $\dim_{\mathbb{C}}H^0(M_g,\phi^{-1}T^{(1,0)}\mathbb{CP}^n)$ more closely and finding an improved estimate for its dimension by using the connection-preserving isomorphism (\ref{pullback-iso-equation}).

Considering the harmonic sequence (\ref{Adiagram}) above, we say a full harmonic map $\phi:M \rightarrow \mathbb{CP}^n$ has $directrix$ $(f,\rho)$ if $\phi=G'(G'(\dots G'(G'(f)) \dots))$ where $G'$ is applied $\rho$ times to a full holomorphic map $f:M \rightarrow \mathbb{CP}^n$ and $\rho \in \{0,1,\dots,n \}$ \cite{BurWood, EellsWood}. Given a harmonic map $\phi:M \rightarrow \mathbb{CP}^n$ with directrix $(f,\rho)$ then by (\ref{pullback-iso-equation}) we have the following decomposition into complex vector bundles:
\begin{align*}
\phi^{-1}T^{(1,0)}\mathbb{CP}^n &\cong L(\uf_{\rho},\uf_{\rho}^{\bot})=L(\uf_{\rho},\uf_{0} \oplus \uf_{1} \oplus \dots \oplus \uf_{\rho-1} \oplus \uf_{\rho+1} \oplus \dots \oplus \uf_{n}) \\   
&\cong L(\uf_{\rho},\uf_{0}) \oplus \dots \oplus L(\uf_{\rho},\uf_{\rho-1}) \oplus L(\uf_{\rho},\uf_{\rho+1}) \oplus \dots \oplus L(\uf_{\rho},\uf_{n})\\
&= A_- + A_+,
\end{align*}
where
$$A_{-}=\sum_{j=0}^{\rho-1} L(\uf_{\rho},\uf_{j}), \qquad A_{+}=\sum_{j=\rho+1}^{n} L(\uf_{\rho},\uf_{j}).$$
As usual, we give $L(\uf_{\rho},\uf_{\rho}^{\bot})$ the connection $^L\nabla$ induced by those on $\uf_{\rho}$ and $\uf_{\rho}^{\bot}$ given by
\begin{equation}\label{L-connection-equ}
( ^{L}\nabla_{\frac{\partial}{\partial \bar{z}}} u)(s)= {^{f_{\rho}^{\bot}}\nabla_{\frac{\partial}{\partial \bar{z}}}} u(s)-u( {}^{f_{\rho}}\nabla_{\frac{\partial}{\partial \bar{z}}} s),
\end{equation}
where ${}^{f_{\rho}^{\bot}}\nabla$ and ${}^{f_{\rho}}\nabla$ are the connections defined in (\ref{connection-subbundles-def}) on $\uf_{\rho}^{\bot}$ and $\uf_{\rho}$ respectively, $u \in \Gamma( L(\uf_{\rho},\uf_{\rho}^{\bot}))$ and $s \in \Gamma(\uf_{\rho})$.
We then give $L(\uf_{\rho},\uf_{\rho}^{\bot})$ the Koszul--Malgrange holomorphic structure from the $(0,1)$-part of that connection $^L\nabla$.
\begin{lemma}\label{A-holo-subbundles-lemma}
Let $L(\uf_{\rho},\uf_{\rho}^{\bot})$ be the holomorphic vector bundle over Riemann surface $M_g$ defined above, then $A_+$ and $A_-$ are both holomorphic subbundles of $L(\uf_{\rho},\uf_{\rho}^{\bot})$. 
\end{lemma}
\begin{proof}
It suffices to show that both $A_+$ and $A_-$ are closed under $ ^L\nabla_{\frac{\partial}{\partial \bar{z}}}$. To show this we let $u \in \Gamma(A_+)$ and $s \in \Gamma(\uf_{\rho})$, then by (\ref{L-connection-equ}), 
$$( ^{L}\nabla_{\frac{\partial}{\partial \bar{z}}} u)(s)= {^{f_{\rho}^{\bot}}\nabla_{\frac{\partial}{\partial \bar{z}}}} u(s)-u( ^{f_{\rho}}\nabla_{\frac{\partial}{\partial \bar{z}}} s).$$
As $u(s) \in \Gamma (\sum_{j=\rho+1}^{n} \uf_{j})$ then $(\partial / \partial\bar{z})u(s) \in \Gamma (\sum_{j=\rho}^{n} \uf_{j})$ by (\ref{Adiagram}), indeed $\partial/\partial\bar{z}$ takes sections of $\uf_i$ to sections of $\uf_i \oplus \uf_{i-1}$ for all $i$.
Therefore we have 
$${^{f_{\rho}^{\bot}}\nabla_{\frac{\partial}{\partial \bar{z}}}} u(s)=\pi_{f_{\rho}^{\bot}}\frac{\partial}{\partial \bar{z}}u(s) \in \Gamma(\sum_{j=\rho+1}^{n} \uf_{j}).$$
Also, as $^{f_{\rho}}\nabla_{\frac{\partial}{\partial \bar{z}}} s \in \Gamma(\uf_{\rho})$ and since $u \in \Gamma(A_+)=\Gamma(\sum_{j=\rho+1}^{n} L(\uf_{\rho},\uf_{j}))$ then $u( {}^{f_{\rho}}\nabla_{\frac{\partial}{\partial \bar{z}}} s) \in \Gamma(\sum_{j=\rho+1}^{n} \uf_{j})$ as required.
A similar argument can be made for $A_-$.
\end{proof}
For each $j \in \{\rho+1,\dots,n\}$, $L(\uf_{\rho},\uf_{j})$ is a complex subbundle of $L(\uf_{\rho},\uf_{\rho}^{\bot})$ and so can be given the induced (subbundle) holomorphic structure, i.e. that with $\bar{\partial}$-operator given by $\pi_{f_{j}} {^{L}\nabla_{\frac{\partial}{\partial \bar{z}}}}$. Using this we can give $A_+$ a second `direct sum' holomorphic structure $\bar{\partial}_{\text{sum}}$ defined by
\begin{equation}\label{d_sum-holo-structure-equation}
\bar{\partial}_{\text{sum}}(\sigma) = \sum_{j=\rho+1}^{n}\pi_{f_{j}} {^{L}\nabla_{\frac{\partial}{\partial \bar{z}}}}(\sigma_j)
\end{equation}
for $\sigma=\sigma_{\rho+1}+\sigma_{\rho+2}+ \cdots +\sigma_{n} \in \Gamma (A_+)$ and $\sigma_j \in \Gamma(L(\uf_{\rho},\uf_{j}))$ for all $j \in \{\rho+1,\dots,n\}$.
\begin{lemma}\label{holo-sum-lemma}
Let $A_+$ be the holomorphic bundle over a compact Riemann surface $M_g$ defined above equipped with the holomorphic structure $\bar{\partial}_{\text{sum}}$ and let the complex subbundle $L(\uf_{\rho}, \uf_{j})$ of $L(\uf_{\rho},\uf_{\rho}^{\bot})$ be equipped with the induced (subbundle) holomorphic structure as above. Then
\begin{enumerate}[label={\upshape(\roman*)}]
\item $H^0(M_g, A_+)=\sum_{j=\rho+1}^{n} H^0(M_g,L(\uf_{\rho}, \uf_{j})),$
\item $H^1(M_g, A_+)=\sum_{j=\rho+1}^{n} H^1(M_g,L(\uf_{\rho}, \uf_{j})).$
\end{enumerate}
\end{lemma}
\begin{proof}
\begin{enumerate}[label={\upshape(\roman*)}]
\item Let $\sigma \in \Gamma (A_+)$ then $\sigma$ may be decomposed uniquely as $\sigma=\sigma_{\rho+1}+\sigma_{\rho+2}+ \cdots +\sigma_{n}$ for $\sigma_j \in \Gamma(L(\uf_{\rho},\uf_{j}))$ for all $j \in \{\rho+1,\dots,n\}$. As $\bar{\partial}_{\text{sum}}\vert_{L(\uf_{\rho},\uf_{j})}=\pi_{f_{j}} {^{L}\nabla_{\frac{\partial}{\partial \bar{z}}}}$ then for each $j$, $L(\uf_{\rho},\uf_{j})$ is closed under $\bar{\partial}_{\text{sum}}$ i.e. is a holomorphic subbundle of $A_+$. We have 
\begin{align*}
\bar{\partial}_{\text{sum}}(\sigma) &=\bar{\partial}_{\text{sum}}(\sigma_{\rho+1})+\bar{\partial}_{\text{sum}}(\sigma_{\rho+2})+\cdots+\bar{\partial}_{\text{sum}}(\sigma_{n}) \\
&=\pi_{f_{\rho+1}} {^{L}\nabla_{\frac{\partial}{\partial \bar{z}}}}(\sigma_{\rho+1})+\pi_{f_{\rho+2}} {^{L}\nabla_{\frac{\partial}{\partial \bar{z}}}}(\sigma_{\rho+2})+\cdots+\pi_{f_{n}} {^{L}\nabla_{\frac{\partial}{\partial \bar{z}}}}(\sigma_{n}).
\end{align*}
Therefore $\sigma \in H^0(M_g,A_+)$ if and only if $\sigma_j \in H^0(M_g,L(\uf_{\rho},\uf_{j}))$ for each $j \in \{\rho+1,\dots,n\}$.
\item Using Serre duality \cite[Theorem 9]{Gunning} we have $H^1(M_g, A_+) \cong H^0(M_g, A_+^* \otimes T^*M_g),$ then (ii) follows from (i). 
\end{enumerate}
\end{proof}
\begin{proposition}\label{dimH^0 geq dimH^0-dimH^1-prop}
Let $M_g$ be a compact Riemann surface of genus $g$, $\phi: M_g \rightarrow \mathbb{CP}^n$ a harmonic map with directrix $(f, \rho)$, let $A_{+}=\sum_{j=\rho+1}^{n} L(\uf_{\rho},\uf_{j})$ be equipped with the holomorphic structure $\bar{\partial}_{\text{sum}}$ and let the complex subbundles $L(\uf_{\rho}, \uf_{j})$ for $j=\rho+1,\dots,n$ of $L(\uf_{\rho},\uf_{\rho}^{\bot})$ be each equipped with the induced (subbundle) holomorphic structure. Then 
$$\dim H^0(M_g,\phi^{-1}T^{(1,0)}\mathbb{CP}^n) \geq \sum_{j=\rho+1}^n \dim H^0(M_g, L(\uf_{\rho},\uf_{j}))-\dim H^1(M_g, L(\uf_{\rho},\uf_{j})).$$
\end{proposition}
\begin{proof}
Recall from (\ref{pullback-iso-equation}) that $\phi^{-1}T^{(1,0)}\mathbb{CP}^n \cong L(\uf_{\rho},\uf_{\rho}^{\bot})$ is an isomorphism of holomorphic vector bundles where $\phi^{-1}T^{(1,0)}\mathbb{CP}^n$ and $L(\uf_{\rho},\uf_{\rho}^{\bot})$ are given the Koszul--Malgrange holomorphic structures defined from their respective connections. By Lemma \ref{A-holo-subbundles-lemma}, $A_+$ and $A_-$ are holomorphic subbundles of $L(\uf_{\rho},\uf_{\rho}^{\bot})$, so we have
\begin{align}\label{H^0-H^1-equation}
\dim H^0(M_g,\phi^{-1}T^{(1,0)}\mathbb{CP}^n) &= \dim H^0(M_g, A_+) + \dim H^0(M_g, A_-) \nonumber\\
&\geq \dim H^0(M_g, A_+) \nonumber \\
&\geq \dim H^0(M_g, A_+) - \dim H^1(M_g, A_+),
\end{align}
as $\dim H^0(M_g, A_-) \geq 0$ and $ \dim H^1(M_g, A_+) \geq 0$. By Corollary \ref{RiemannRoch-cor} the right-hand side of 
(\ref{H^0-H^1-equation}) is independent of choice of the holomorphic structure on $A_+$. Using this, we replace the holomorphic structure of $A_+$ induced by the Koszul--Malgrange holomorphic structure of $L(\uf_{\rho},\uf_{\rho}^{\bot})$ with the holomorphic structure defined by (\ref{d_sum-holo-structure-equation}). Applying Lemma \ref{holo-sum-lemma} to $\dim H^0(M_g, A_+) - \dim H^1(M_g, A_+)$ with $A_+$ equipped with the holomorphic structure $\bar{\partial}_{\text{sum}}$, we have
\begin{align*}
\dim H^0(M_g,\phi^{-1}T^{(1,0)}\mathbb{CP}^n) &\geq \dim H^0(M_g, A_+) - \dim H^1(M_g, A_+) \\
&= \sum_{j=\rho+1}^n \{\dim H^0(M_g, L(\uf_{\rho},\uf_{j}))-\dim H^1(M_g, L(\uf_{\rho},\uf_{j}))\}.
\end{align*}
\end{proof}
By noting that the degree of $\phi$ is minus the  first Chern class $c_1$ of the bundle $\underline{\phi}$ (see \cite[Lemma 5.1]{BurWood}), we have from \cite[p.\ 246]{EellsWood}, given a harmonic map $\phi:M_g \rightarrow \mathbb{CP}^n$ with directrix $(f,\rho)$ then
\begin{equation}\label{chern}
c_1(\phi)=\sum_{\alpha=0}^{\rho-1}r_{\alpha}-\deg(f)+\rho(2-2g).
\end{equation}
We deduce the following.
\begin{theorem}\label{general-index-prop}
Let $M_g$ be a compact Riemann surface of genus $g$, $\phi: M_g \rightarrow \mathbb{CP}^n$ be a full non-$\pm$-holomorphic complex isotropic map with directrix $(f, \rho)$, and let $r_\alpha$ be the $(\alpha+1)$st total ramification index of $f$, then
\begin{equation}\label{general-index-equation}
\Index(\phi) \geq (n+1)\deg(f)-\sum_{\alpha=0}^{\rho -1} (n-\alpha) r_{\alpha} + (2n\rho-\rho^2+2\rho-n)(g-1).
\end{equation}
\end{theorem}
\begin{proof}
For each $j \in \{ \rho+1, \dots, n\}$, we have using (\ref{chern}) that
$$c_1(L(\uf_{\rho},\uf_{j})) = c_1(\uf_{\rho}^* \otimes \uf_{j}) =-c_1(\uf_{\rho})+c_1(\uf_{j})=-(j-\rho)(2g-2)+\sum_{\alpha=\rho}^{j-1}r_{\alpha}.$$
Therefore for each $j \in \{ \rho+1, \dots, n\}, \,$ Theorem \ref{RiemannRoch} (Riemann-Roch) gives
\begin{align*}
\dim H^0(M_g, L(\uf_{\rho},\uf_{\alpha}))-\dim H^1(M_g, L(\uf_{\rho},\uf_{\alpha})) &=c_1(L(\uf_{\rho},\uf_{j}))+1-g \\
&=-(2j-2\rho+1)(g-1)+\sum_{\alpha=\rho}^{j-1}r_{\alpha}.
\end{align*}
Using this together with Proposition \ref{dimH^0 geq dimH^0-dimH^1-prop} we have
\begin{align}\label{dimH^0-inequality-1}
\dim H^0(M_g,&\phi^{-1}T^{(1,0)}\mathbb{CP}^n) \geq \sum_{j=\rho+1}^n \{\dim H^0(M_g, L(\uf_{\rho},\uf_{j}))-\dim H^1(M_g, L(\uf_{\rho},\uf_{j}))\} \nonumber\\
&=\sum_{j=\rho+1}^n \{c_1(L(\uf_{\rho},\uf_{j}))+1-g\} \nonumber\\
&=\Big((n-\rho)(2\rho-1)-n(n+1)+\rho(\rho+1)\Big)(g-1)+\sum_{\alpha=\rho}^{n-1}(n-\alpha)r_{\alpha}.
\end{align}
From \cite[p.\ 271]{GriffHarris} we have a useful relation involving the total ramification indices $r_{\alpha}$:
\begin{equation}\label{ramrelation}
\sum_{\alpha=0}^{n-1} (n-\alpha) r_{\alpha}=(n+1)\deg(f)+n(n+1)(g-1).
\end{equation}
We split this sum so that we have 
\begin{equation}\label{splitramrelation}
\sum_{\alpha=\rho}^{n-1} (n-\alpha) r_{\alpha}=(n+1)\deg(f)+n(n+1)(g-1)-\sum_{\alpha=0}^{\rho -1} (n-\alpha) r_{\alpha}.
\end{equation}
By substituting (\ref{splitramrelation}) into (\ref{dimH^0-inequality-1}) we have the following:
\begin{align*}
\dim H^0(M_g,&\phi^{-1}T^{(1,0)}\mathbb{CP}^n) \geq (n+1)\deg(f)-\sum_{\alpha=0}^{\rho -1} (n-\alpha) r_{\alpha} + \\
& \quad \Big((n-\rho)(2\rho-1)-n(n+1)+\rho(\rho+1)+n(n+1)\Big)(g-1) \\
&=(n+1)\deg(f)-\sum_{\alpha=0}^{\rho -1} (n-\alpha) r_{\alpha} + (2n\rho-\rho^2+2\rho-n)(g-1).
\end{align*}
By Lemma \ref{index>dim} the theorem is proven.
\end{proof}
\begin{corollary}\label{general-index-corollary}
Let $M_g$ be a compact Riemann surface of genus $g$, $\phi: M_g \rightarrow \mathbb{CP}^n$ be a full non-$\pm$-holomorphic complex isotropic map with directrix $(f, \rho)$, and let $r_\alpha$ be the $(\alpha+1)$st total ramification index of $f$, then
$$\Index(\phi) \geq (n+1)\deg(\phi)+(n+\rho^2)(1-g)+\sum_{\alpha=0}^{\rho-1}(\alpha+1)r_{\alpha}.$$
\end{corollary}
\begin{proof}
By (\ref{chern}) we have
\begin{equation}\label{degreesrel}
\deg (f)= \deg (\phi) +\sum_{\alpha=0}^{\rho-1} r_{\alpha}-\rho(2g-2).
\end{equation}
Substituting this into (\ref{general-index-equation}) gives the desired result.
\end{proof}
\begin{remark}\label{theorem improvement remark}
Theorem \ref{general-index-prop} is an improvement on known estimates (Proposition \ref{estimate}) if and only if
$$\sum_{\alpha=0}^{\rho-1} (\alpha+1) r_{\alpha}>\rho^2(g-1).$$
Further, Theorem \ref{general-index-prop} is an improvement on Proposition \ref{estimate} for harmonic maps $\phi:S^2 \rightarrow \mathbb{CP}^n$ and $\phi:M_1 \rightarrow \mathbb{CP}^n$ with directrix $(f,\rho)$ by amounts ${\rho^2+\sum_{\alpha=0}^{\rho-1}(\alpha+1)r_{\alpha}}$ and ${\sum_{\alpha=0}^{\rho-1}(\alpha+1)r_{\alpha}}$ respectively.
\end{remark}
\begin{corollary}\label{Sphere - CP2 Corollary}
Let $\phi:S^2 \rightarrow \mathbb{CP}^2$ be a full non-$\pm$-holomorphic harmonic map with directrix $(f,1)$ and let $r_0$ be the first total ramification index of f. Then
$$\Index(\phi) \geq 3\deg(f)-2r_0-3 = 3\deg(\phi)+r_0+3.$$
\end{corollary}
\begin{proof}
This follows immediately from Theorem \ref{general-index-prop} and Corollary \ref{general-index-corollary} for $n=2$, $\rho=1$ and $g=0$.
\end{proof}
\begin{corollary}\label{Torus - CP2 Corollary}
Let $\phi:M_1 \rightarrow \mathbb{CP}^2$ be a full non-$\pm$-holomorphic complex isotropic harmonic map with directrix $(f,1)$ and let $r_0$ be the first total ramification index of f. Then
$$\Index(\phi) \geq 3\deg(f)-2r_0 = 3\deg(\phi)+r_0.$$
\end{corollary}
\begin{proof}
This follows immediately from Theorem \ref{general-index-prop} and Corollary \ref{general-index-corollary} for $n=2$, $\rho=1$ and $g=1$.
\end{proof}
\begin{remark}\label{corollary improvement remark}
Corollary \ref{Sphere - CP2 Corollary} and Corollary \ref{Torus - CP2 Corollary} are improvements for a harmonic map $\phi:S^2 \rightarrow \mathbb{CP}^2$ and $\phi:M_1 \rightarrow \mathbb{CP}^2$ with directrix $(f,1)$ by the amount $1+ r_0$ and $r_0$, respectively.
\end{remark}

\section{Examples}
We present examples of harmonic maps for genus $0,1$, and higher genera for which the known estimates on the index are improved by Corollary \ref{Sphere - CP2 Corollary}, Corollary \ref{Torus - CP2 Corollary} and Theorem \ref{general-index-prop} respectively.
\subsection{Genus 0}
We define the family of maps $\eta_k:S^2\rightarrow S^2$ given by $\eta_k(z)=z^k$ for $k\in \mathbb{Z}$: all $\eta_k$ are $\pm$-holomorphic.
\begin{example}cf. \cite[Example 8.1]{EellsWood}
Let $F:S^2 \rightarrow \mathbb{C}^3\setminus\{0\}$, where $F(z)=(1,z,z^2)$. Then let $f=[F]:S^2 \rightarrow \mathbb{CP}^2$, so $f(z)=[1,z,z^2]$ and is a full holomorphic map with $\deg(f)=2$ and $r_0=0$. Following $\S 2.1$ we have that $G'(f):S^2 \rightarrow \mathbb{CP}^2$ is a full non-$\pm$-holomorphic (complex isotropic) harmonic map of degree $0$ and directrix $(f,1)$.
For each $k\in\mathbb{N}$ the composition $f \circ \eta_k:S^2 \rightarrow \mathbb{CP}^2$ gives a full holomorphic map with $\deg(f \circ \eta_k)=2k$ and $r_0=2(k-1)$. For each $k\in\mathbb{N}$, $G'(f \circ \eta_k):S^2 \rightarrow \mathbb{CP}^2$ gives a full non-$\pm$-holomorphic harmonic map of degree $0$ and directrix $(f \circ \eta_k,1)$. By Corollary \ref{Sphere - CP2 Corollary}, $\Index (G'(f \circ \eta_k)) \geq 3\deg (f \circ \eta_k) - 2 r_0-3=6k-4(k-1)-3=2k+1.$ By Remark \ref{corollary improvement remark}, for each $k\in\mathbb{N}$, Corollary \ref{Sphere - CP2 Corollary} improves the estimate in \cite{EellsWood} (see Proposition \ref{estimate} above) by $2k-1$.
\end{example}
\begin{example}cf. \cite[Example 8.2]{EellsWood}\label{S1 - deg 1 - Example}
Let $F:S^2 \rightarrow \mathbb{C}^3\setminus\{0\}$, where $F(z)=(1,z+z^3,z^2)$. Then let $f=[F]:S^2 \rightarrow \mathbb{CP}^2$, so $f(z)=[1,z+z^3,z^2]$ and is a full holomorphic map with $\deg(f)=3$ and $r_0=0$. Again following $\S 2.1$ we have that $G'(f):S^2 \rightarrow \mathbb{CP}^2$ is a full non-$\pm$-holomorphic (complex isotropic) harmonic map of degree $1$ and directrix $(f,1)$. 
For each $k\in\mathbb{N}$ the composition $f \circ \eta_k:S^2 \rightarrow \mathbb{CP}^2$ gives a full holomorphic map with $\deg(f \circ \eta_k)=3k$ and $r_0=2(k-1)$. For each $k\in\mathbb{N}$, $G'(f \circ \eta_k):S^2 \rightarrow \mathbb{CP}^2$ gives a full non-$\pm$-holomorphic harmonic map of degree $k$ and directrix $(f \circ \eta_k,1)$. By Corollary \ref{Sphere - CP2 Corollary}, $\Index (G'(f \circ \eta_k)) \geq 3\deg (f \circ \eta_k) - 2 r_0-3=9k-4(k-1)-3=5k+1.$ By Remark \ref{corollary improvement remark}, for each $k\in\mathbb{N}$, Corollary \ref{Sphere - CP2 Corollary} improves the estimate in \cite{EellsWood} (see Proposition \ref{estimate} above) by $2k-1$.
\end{example}
\subsection{Genus 1}
Let $M_1$, $M'_1$ be tori, i.e. compact Riemann surfaces of genus $1$ and $\psi:M_1\rightarrow M'_1$ a holomorphic covering map of degree $k$.
\begin{example}
Let $f: M'_1 \rightarrow \mathbb{CP}^2$ be the degree 5 full holomorphic map with first total ramification index $4$ constructed in \cite[Lemma 8.7]{EellsWood}. The composition $f \circ \psi$ is a full holomorphic map with $\deg(f \circ \psi)=5k$ and $r_0=4k$. As an application of Corollary \ref{Torus - CP2 Corollary} let $G'(f \circ \psi): M_1 \rightarrow \mathbb{CP}^2$ be the degree $k$ harmonic non-$\pm$-holomorphic map with directrix $(f \circ \psi,1)$ then $\Index (G'(f \circ \psi)) \geq 3\deg (f \circ \psi) - 2 r_0=15k-8k=7k.$ By Remark \ref{corollary improvement remark}, for each $k\in\mathbb{N}$, Corollary \ref{Torus - CP2 Corollary} improves the estimate in \cite{EellsWood} (see Proposition \ref{estimate} above) by $4k$.
\end{example}

\subsection{Higher genera}
Let $M_g$ be a compact Riemann surface of genus $g>1$. 
\begin{example}
By \cite[Theorem 8.10]{EellsWood} there exist full non-$\pm$-holomorphic complex isotropic harmonic maps $\phi:M_g \rightarrow \mathbb{CP}^2$ of degree $k > g$. Indeed there exist holomorphic maps $h:M_g\rightarrow\mathbb{CP}^1$ of all degrees $k > g$. Composing such a map with the full harmonic map of degree $1$ with directrix $(f,1)$ from Example \ref{S1 - deg 1 - Example} gives a full non-$\pm$-holomorphic complex isotropic harmonic map of degree $k > g$. This is the Gauss transform of the full holomorphic map $f \circ h$ which has degree $3k$. From (\ref{degreesrel}) $r_0 = 2k+2g-2 > g-1$: Therefore by Remark \ref{theorem improvement remark}, Theorem \ref{general-index-prop} improves the estimate in \cite{EellsWood} (see Proposition \ref{estimate} above) for all these maps, giving examples in all degrees $>g$.
\end{example}

\end{document}